\documentclass[twoside,reqno]{amsart}
\usepackage{amsfonts,amsmath,amscd,amsthm,amssymb}
\usepackage{mathtools}
\usepackage[mathscr]{euscript} 
\usepackage{graphics}
\usepackage{url}
\usepackage{wrapfig}
\usepackage{lscape}
\usepackage{rotating}
\usepackage{epsfig}
\usepackage{cite}
\usepackage{color}
\usepackage{booktabs}
\usepackage{multirow}
\usepackage{lipsum,multicol}

\usepackage{tikz}
\usetikzlibrary{patterns}

\usepackage{amssymb}
\usepackage{bbm}
\usepackage{multirow}

\newtheorem{theorem}{Theorem}[section]
\newtheorem{lemma}[theorem]{Lemma}

\theoremstyle{definition}

\theoremstyle{remark}
\newtheorem{remark}[theorem]{Remark}

\numberwithin{equation}{section}
\numberwithin{table}{section}

\def\ho{h^{k+1}} 

\def\jj{j = 1, \ldots, N}

\def\ih{\mathcal{I}_h}
\def\jump#1{[\![{#1}]\!]} 
\def\mean#1{\{\!\!\{{#1}\}\!\!\}} 

\def\normI#1{\|{#1}\|_{L^2(I)}}  
\def\intc#1{\int_{I_j}{#1}\mathrm{dx} } 
\def\intid#1{\int_{I}{#1}\mathrm{dx} } 

\begin{document}

\title[Energy conserving ultra-weak DG]{An energy-conserving ultra-weak discontinuous Galerkin method for 
the generalized Korteweg-De Vries equation}

\author{Guosheng Fu}
\address{Division of
Applied Mathematics, Brown University, Providence, RI
 02912}
\curraddr{}
\email{guosheng\_fu@brown.edu}

\author{Chi-Wang Shu}
\address{Division of
Applied Mathematics, Brown University, Providence, RI
 02912}
\curraddr{}
\email{shu@dam.brown.edu}
\thanks{The research of the second author was supported
  by ARO grant W911NF-15-1-0226 and NSF grant DMS-1719410.}



\dedicatory{}


\begin{abstract}
We propose an energy-conserving ultra-weak discontinuous Galerkin (DG) method
for the generalized Korteweg-De Vries (KdV) equation in one dimension.
Optimal a priori error estimate of order $k+1$ is obtained for the semi-discrete scheme 
for the KdV equation without the
convection term on general nonuniform meshes when polynomials 
of degree $k\ge 2$ are used.
We also numerically observe optimal convergence of the method for the KdV equation with linear or nonlinear 
convection terms.

It is numerically observed for the new method to have a superior performance for long-time 
simulations over existing DG methods.

\end{abstract}

\maketitle
\section{Introduction}
In this paper, we present an energy-conserving discontinuous Galerkin
(DG) method for the 
generalized Korteweg-De Vries (KdV) equation:
 \begin{align}
 \label{eq:kdv}
 u_t + f(u)_x + \epsilon u_{xxx} =&\; 0, \quad (x,t) \in (0, 1)\times (0,T],\\
 u(x,0) = &\; u_0(x),
\end{align}
with periodic boundary conditions, 
where $f$ is a (nonlinear) function of $u$, and
$\epsilon\not = 0$ is a constant parameter.

A vast amount of literature can be found on the numerical approximation of 
the above equation, c.f. references cited in \cite{Bona13}. 
All types of numerical methods, including finite difference, finite element, finite
volume and spectral methods have their proponents.  Here, we will confine our attention
in finite element methods, in particular, DG methods.
The DG methods, c.f. \cite{Cockburn00}, belong to a class of finite element methods 
using discontinuous piecewise polynomial
spaces for both the numerical solution and the test functions. 
They allow arbitrarily unstructured meshes, and have compact stencils. Moreover, they
easily accommodate arbitrary $h$-$p$ adaptivity. 

Various DG methods can be applied to solve \eqref{eq:kdv}, including the LDG methods \cite{YanShu02,XuShu10, Chen15},
the direct DG method \cite{Yi13}, and the ultra-weak DG method \cite{ChengShu08,Bona13}. 
Among these cited references, the methods \cite{Bona13, Yi13,Chen15} are energy conserving, while the methods 
\cite{YanShu02, ChengShu08,XuShu10} are energy dissipative.
Energy conserving DG methods were numerically shown \cite{Bona13, Yi13} to outperform their dissipative counterparts for long time simulations.
However, the above cited energy-conserving DG methods are numerically 
observed to be suboptimally convergent for odd polynomial degrees,
while optimal error estimates have only be proven for even polynomial degrees on {\it uniform} meshes 
\cite{Bona13, Chen15} for the equation \eqref{eq:kdv} without 
the convection term, i.e. $f(u)=0$. 
To the best of our knowledge, 
no optimally convergent, energy-conserving DG methods using odd 
polynomial degrees have been available for the equation \eqref{eq:kdv}, 
even numerically, in the literature.

We fill this gap by presenting a new energy conserving DG method that is proven to be optimally convergent for all polynomial degrees 
$k\ge 2$ for the equation \eqref{eq:kdv} without the
convection term ($f(u)=0$) on general nonuniform meshes, and is 
numerically shown to be optimally convergent for the equation \eqref{eq:kdv} with linear or nonlinear convection terms.
Our method is based on the ultra-weak DG formulation \cite{ChengShu08}, and on the idea of 
{\it doubling the unknowns} in our recent work on 
optimal energy conserving DG methods for linear symmetric hyperbolic systems \cite{FuShu18}.
We point out that the choice of ultra-weak DG formulation is not essential, one
can equivalently apply the idea to the LDG method \cite{XuShu10}, and obtain an optimal energy conserving LDG method
for the equation \eqref{eq:kdv}.
We specifically note that the polynomial degree $k\ge 2$ is required for the third order equation \eqref{eq:kdv} in 
the ultra-weak DG formulation, while the LDG method \cite{XuShu10} do not have such order restriction.
The ultra-weak DG method is computationally more efficient over the LDG method due to its compactness, c.f. \cite{ChengShu08}.

The rest of the paper is organized as follows. 
In Section \ref{sec:1d}, we present and analyze
the semi-discrete energy-conserving DG method for  the equation \eqref{eq:kdv}.
Numerical results are reported in Section \ref{sec:num}. Finally, we conclude in 
Section \ref{sec:conclude}.

\section{The numerical scheme}\label{sec:1d}
In this section, we present the energy-conserving DG method
for the generalized KdV equation \eqref{eq:kdv}.
We first present and analyze the scheme for \eqref{eq:kdv} without 
the convection term ($f=0$). An optimal error estimate is obtained for this case.
We then treat the nonlinear convection term using the square entropy-conserving flux \cite{Tadmor87}.

\subsection{Notation and definitions in the one-dimensional case}
In this subsection, we shall first introduce some notation and definitions in the one-dimensional case,
which will be used throughout this section.
\subsubsection{The meshes}
Let us denote by $\ih$ a tessellation of the computational interval $I = [0, 1]$,
consisting of cells $I_j = (x_{j-\frac12},x_{j+\frac12})$ with $1
\le j \le N$, where
$$
0 = x_\frac12 < x_\frac32 < \cdots < x_{N+\frac12} = 1.
$$
The following standard notation of DG methods will be used.
Denote $x_j =
(x_{j-\frac12} + x_{j+\frac12})/2$, $h_j = x_{j+\frac12} -
x_{j-\frac12}$, $h = \max_jh_j$, and $\rho = \min_jh_j$.
The mesh is assumed to be regular in the sense that $h / \rho$ 
is always bounded during mesh refinements, namely, there
exists a positive constant $\gamma$ such that $\gamma h \le \rho \le h$.
We denote by $p_{j+\frac12}^-$ and $p_{j+\frac12}^+$ the values
of $p$ at the discontinuity point $x_{j+\frac12}$, from the left
cell, $I_j$, and from the right cell, $I_{j+1}$, respectively.
In what follows, we employ $\jump p = p^+ - p^-$ and $\mean p = \frac12(p^+ +
p^-)$ to represent the jump and the mean value of $p$ at each element
boundary point. The following discontinuous piecewise polynomials space is chosen
as the finite element space:
\begin{align}
\label{space-1d}
V_h \equiv V_h^{k} = \left\{v \in L^2(I): v|_{I_j} \in P^{k}(I_j),
~~ j = 1, \ldots, N \right\}, 
\end{align}
where $P^{k}(I_j)$ denotes the set of polynomials of degree up to $k\ge 2$ defined on the cell $I_j$.

\subsubsection{Function spaces and norms}
Denote $H^s(I)$ as the space of $L^2$ functions on $I$ whose $s$-th 
derivative is also an $L^2$ function.
Denote $\|\cdot\|_{I_j}$ the standard $L^2$-norm on the cell $I_j$, and 
$\|\cdot\|_I$ the $L^2$-norm on the whole interval.

\newcommand{\bld}[1]{\boldsymbol{#1}}
\subsection{Case $f=0$}
\label{sub:c0}
In this subsection, we consider the KdV equation without convection effect ($f=0$):
\begin{align}
\label{advection1d}
 u_t +  \epsilon u_{xxx} & =0, &&\hspace{-2.8cm}  (x,t)\in I\times (0,T] ,
\end{align}
with a smooth periodic initial condition $u(x,0) = u_0(x)$ for $x\in I$, and 
a periodic boundary condition.

To derive the optimal energy-conserving DG method for this equation, 
we shall following the idea in our recent work \cite{FuShu18} on optimal energy-conserving DG methods
for linear symmetric hyperbolic systems by 
doubling the unknowns with the introduction of an auxiliary {\it zero} function $\phi(x,t) = 0$.
Hence, we consider the following $2\times 2$ system:
 \begin{subequations}
 \label{aux-adv}
\begin{align}
 u_t + \epsilon u_{xxx} & =0, &&\hspace{-2.8cm}  (x,t)\in I\times (0,T] ,\\
 \phi_t -\epsilon \phi_{xxx} & =0, &&\hspace{-2.8cm}  (x,t)\in I\times (0,T] ,
\end{align}
with initial condition $u(x,0)=u_0(x)$ and $\phi(x,0)=0$.
\end{subequations}

We use the ultra-weak DG method of Cheng and Shu \cite{ChengShu08} to discretize the above equations.
The semi-discrete DG method for \eqref{aux-adv} is as follows.
Find, for any time $t \in (0, T]$,
the unique function $(u_h, \phi_h) = (u_h(t), \phi_h(t))
\in V_h^k\times V_h^k$ such that
 \begin{subequations}
 \label{scheme:adv1d}
\begin{align} \label{scheme:adv1d-1}
\intc{(u_h)_tv_h} +D_{u,j}((u_h, \phi_h), v_h) =&\; 0,\\
\label{scheme:adv1d-2}
\intc{(\phi_h)_t\psi_h} 
-D_{\phi,j}((u_h, \phi_h), \psi_h)=&\;0,
\end{align}
holds for all $(v_h, \psi_h)\in V_h^k\times V_h^k$ and all $\jj$,
where the bilinear forms
\begin{align}
\label{op-uxxx}
 D_{u,j} = &
  -\epsilon\intc{u_h (v_h)_{xxx}}\\
&\;+\epsilon\widehat{u}_h (v_h)_{xx}^-\Big|_{j+1/2}
-\epsilon\widehat{u}_{h,x} (v_h)_{x}^-\Big|_{j+1/2}
+\epsilon\widehat{u}_{h,xx} (v_h)^-\Big|_{j+1/2}\nonumber\\
&\;+\epsilon\widehat{u}_h (v_h)_{xx}^+\Big|_{j-1/2}
-\epsilon\widehat{u}_{h,x} (v_h)_{x}^+\Big|_{j-1/2}
+\epsilon\widehat{u}_{h,xx} (v_h)^+\Big|_{j-1/2}\;\nonumber\\
\label{op-pxxx}
D_{\phi,j} = &
  -\epsilon\intc{\phi_h (\psi_h)_{xxx}}\\
&\;+\epsilon\widehat{\phi}_h (\psi_h)_{xx}^-\Big|_{j+1/2}
-\epsilon\widehat{\phi}_{h,x} (\psi_h)_{x}^-\Big|_{j+1/2}
+\epsilon\widehat{\phi}_{h,xx} (\psi_h)^-\Big|_{j+1/2}\nonumber\\
&\;+\epsilon\widehat{\phi}_h (\psi_h)_{xx}^+\Big|_{j-1/2}
-\epsilon\widehat{\phi}_{h,x} (\psi_h)_{x}^+\Big|_{j-1/2}
+\epsilon\widehat{\phi}_{h,xx} (\psi_h)^+\Big|_{j-1/2}\;\nonumber\\
\end{align}
and the numerical fluxes are chosen as follows, c.f. \cite{FuShu18},
\begin{align}
\label{flux1}
\widehat{u}_h = \mean{u_h} +\frac12\jump{\phi_h}, &\quad 
\widehat{\phi}_h = \mean{\phi_h} +\frac12\jump{u_h},\\
\label{flux2}
\widehat{u}_{h,x} = \mean{(u_h)_x} +\frac12\jump{(\phi_h)_x}, &\quad 
\widehat{\phi}_{h,x} = \mean{(\phi_h)_x} +\frac12\jump{(u_h)_x},\\
\label{flux3}
\widehat{u}_{h,xx} = \mean{(u_h)_{xx}} +\frac12\jump{(\phi_h)_{xx}}, &\quad 
\widehat{\phi}_{h,xx} = \mean{(\phi_h)_{xx}} +\frac12\jump{(u_h)_{xx}}.
\end{align}
\end{subequations}

The energy conservation property of this scheme is presented in the following theorem.
\begin{theorem}\label{thm:adv1d}
The  energy 
\[
 E_h(t) = \int_I((u_h)^2+(\phi_h^2))\mathrm{dx}
\]
is conserved by the semi-discrete scheme \eqref{scheme:adv1d}
for all time.
\end{theorem}
\begin{remark}[Modified energy]
We specifically remark here that it is the total energy 
\[
  E_h(t) = \int_I((u_h)^2+(\phi_h^2))\mathrm{dx}
\]
that is conserved, not the quantity $\int_I (u_h)^2 \mathrm{dx}$.
The quantity $\phi_h$ is an approximation to the {\it zero} function.
\end{remark}

\begin{proof}
 By repeatedly integrating by parts and using the periodic boundary condition, we have 
 \begin{align*}
 & \sum_{j=1}^N \left( -\intc{u_h\,(u_h)_{xxx}} \right) \\
 & = \sum_{j=1}^N(\mean{u_h}\jump{(u_h)_{xx}}-\mean{(u_h)_{x}}\jump{(u_h)_x}+
  \mean{(u_h)_{xx}}\jump{u_h})|_{j-1/2}.\nonumber
 \end{align*}
Taking $v_h=u_h$ and $\psi_h=\phi_h$ in the scheme \eqref{scheme:adv1d}, and summing 
over all elements, we get 
\begin{align*}
 \frac12\int_I{(u_h)^2}\mathrm{dx}
 +\Theta = 0,\\
 \frac12\int_I{(\phi_h)^2}\mathrm{dx}
 -\Theta = 0,
\end{align*}
where 
\[
 \Theta = 
\frac12\sum_{j=1}^N(-\jump{\phi_h}\jump{(u_h)_{xx}}+\jump{(\phi_h)_{x}}\jump{(u_h)_x}-
  \jump{(\phi_h)_{xx}}\jump{u_h})|_{j-1/2}.
\]
We get the desired result by adding the above equations, and integrating over $t$.
\end{proof}

Now, we turn to the error estimates of the scheme \eqref{scheme:adv1d}.
We start by introducing a set of projections, similar to the ones used in \cite{ChengShu08}.
We shall use the following left and right generalized Gauss-Radau projections $P_h^{\pm}\in P^k(I_j)$ for $k\ge2$.
\begin{subequations}\label{proj1d}
\begin{align}
\label{pj1}
\intc{P_h^\pm u(x) v_h}  &= \intc{u(x) v_h} && \hspace{-2cm} \forall v_h \in
P^{k-3}(I_j), \\
 (P_h^\pm u)^\pm  &= u^\pm &&\hspace{-2cm} \text{at}~~ x_{j\mp\frac12},\\
 (P_h^\pm u)_x^\pm  &= u_x^\pm &&\hspace{-2cm} \text{at}~~ x_{j\mp\frac12},\\
 (P_h^\pm u)_{xx}^\pm  &= u_{xx}^\pm &&\hspace{-2cm} \text{at}~~ x_{j\mp\frac12},
\end{align}
where the first set of equations \eqref{pj1} vanishes if $k=2$.
The following approximation properties of $P_h^{\pm}$ is well-known, c.f. \cite{ChengShu08},
\begin{align}
\|P_h^\pm u - u\|_{I_j}\le Ch^{k+1}.
\end{align}
\end{subequations}

We shall also use the following coupled projection specifically designed for the DG scheme \eqref{scheme:adv1d}.
For any function $u, \phi\in H^2(I)$, 
we introduce the following coupled auxiliary projection 
 $(P_h^{1,\star} u, P_h^{2,\star} \phi)\in [V_h^k]^2$:
 \begin{subequations}
 \label{proj-adv}
 \begin{align}
 \label{proj-adv1}
  \intc{P_h^{1,\star} u(x) v_h}  &= \intc{u(x) v_h} &&  \forall v_h \in P^{k-3}(I_j),\\
 \label{proj-adv2}
  \intc{P_h^{2,\star} \phi(x) v_h}  &= \intc{\phi(x) v_h}  &&  \forall v_h \in P^{k-3}(I_j),\\
 \label{proj-adv3}
  (\mean {P_h^{1,\star} u_h} +{\frac12}\jump {P_h^{2,\star} \phi_h})\Big|_{j-\frac12} &=  u(x_{j-\frac12}),\\
 \label{proj-adv4}
 ( \mean {P_h^{2,\star} \phi_h} +{\frac12}\jump {P_h^{1,\star} u_h})\Big|_{j-\frac12} &=  \phi(x_{j-\frac12}),\\
 \label{proj-adv5}
  (\mean {(P_h^{1,\star} u_h)_x} +{\frac12}\jump {(P_h^{2,\star} \phi_h)_x})\Big|_{j-\frac12} &=  u_{x}(x_{j-\frac12}),\\
 \label{proj-adv6}
 ( \mean {(P_h^{2,\star} \phi_h)_{x}} +{\frac12}\jump {(P_h^{1,\star} u_h})_x)\Big|_{j-\frac12} &=  \phi_{x}(x_{j-\frac12}),\\
 \label{proj-adv7}
  (\mean {(P_h^{1,\star} u_h)_{xx}} +{\frac12}\jump {(P_h^{2,\star} \phi_h)_{xx}})\Big|_{j-\frac12} &=  u_{xx}(x_{j-\frac12}),\\
 \label{proj-adv8}
 ( \mean {(P_h^{2,\star} \phi_h)_{xx}} +{\frac12}\jump {(P_h^{1,\star} u_h})_{xx})\Big|_{j-\frac12} &=  \phi_{xx}(x_{j-\frac12}),
 \end{align}
 \end{subequations}
 for all $j$.
 
 At a first glance, the projection \eqref{proj-adv} seems to be globally coupled. 
 The following Lemma shows that it is actually an optimal local projection.
 \begin{lemma}
 \label{lemma:proj} Let $k\ge 2$. 
  The projection \eqref{proj-adv} is well-defined, and it satisfies
  \begin{subequations}
  \label{transformX}
  \begin{align}
   P_h^{1,\star} u = \frac12(P_h^+(u+\phi)+P_h^-(u-\phi)),\\
   P_h^{2,\star} \phi = \frac12(P_h^+(u+\phi)-P_h^-(u-\phi)).
  \end{align}
  In particular, it satisfies
  \begin{align}
  \label{approx-proj}
  \| P_h^{1,\star} u -u\|_{I_j}\le Ch^{k+1}, \text{ and }
  \| P_h^{2,\star} \phi -\phi\|_{I_j}\le Ch^{k+1}.
  \end{align}
    \end{subequations}
 \end{lemma}
 \begin{proof}
 The proof is identical to \cite[Lemma 2.6]{FuShu18}, and is omitted for simplicity.
 We specifically mention that the {\it local} structure of the projection is due to the careful choice of the 
 boundary terms (numerical fluxes) in \eqref{proj-adv}.
 \end{proof}

Now, we are ready to state our main result on the error estimate.
\begin{theorem}\label{thm:adv1d:err}
Assume that the exact solution $u$ of \eqref{advection1d} is sufficiently smooth. 
Let $u_h$ be the numerical solution of the semi-discrete DG scheme \eqref{scheme:adv1d} for $k\ge 2$. 
Then for $T >0$ there holds the following
error estimate
\begin{equation} \label{thm:adv-est}
\normI {u(T) - u_h(T)} 
+\normI {\phi_h(T)}
\le C (1+T) \ho,
\end{equation}
where $C$ is independent of $h$ and $T$.
\end{theorem}

\newcommand{\euu}{\varepsilon_u}
\newcommand{\duu}{\delta_u}
\newcommand{\epp}{\varepsilon_\phi}
\newcommand{\dpp}{\delta_\phi}
\newcommand{\puu}{P_h^{1,\star}}
\newcommand{\ppp}{P_h^{2,\star}}
\begin{proof}
The proof is a standard energy argument, again, we skip the details.
In particular, we get the following energy identity by the specific choice of the projection \eqref{proj-adv}:
\begin{align*}
&\frac12\intid{(u_h-P_h^{1,\star}u_h)^2+
(\phi_h-P_h^{1,\star}\phi_h)^2}\\
&=
\intid{(u-P_h^{1,\star}u_h)_t(u_h-P_h^{1,\star}u_h)+
(-P_h^{1,\star}\phi_h)_t(\phi_h-P_h^{1,\star}\phi_h)}.
\end{align*}
\end{proof}
\begin{remark}[$\phi_h$ approximates {\it zero}]
Note that $\phi_h$ is an order $k+1$ approximation to the {\it zero} function.
\end{remark}


\begin{remark}[Other numerical fluxes]
\label{rk:flux}
There are two other choices of numerical fluxes available in the literature\cite{ChengShu08, Bona13}, with both methods devised for the original
equation \eqref{advection1d}.
The flux in \cite{ChengShu08} is given by 
\begin{align}
 \label{flux-cheng}
\widehat{u}_h = {(u_h)^-},\;\;
\widehat{u}_{h,x} = (u_h)_x^+, \;\;
\widehat{u}_{h,xx} = (u_h)_{xx}^+.
\end{align}
The resulting method is optimally convergent, but not energy conserving.
The flux in \cite{Bona13} is given by 
\begin{align}
 \label{flux-xing}
\widehat{u}_h = {(u_h)^-},\;\;
\widehat{u}_{h,x} = \mean{(u_h)_x}, \;\;
\widehat{u}_{h,xx} = (u_h)_{xx}^+.
\end{align}
The resulting method is energy conserving but suboptimal on general non-uniform mesh.
\end{remark}

\subsection{Case $f\not=0$}
Now, we consider the generalized KdV equation \eqref{eq:kdv}.
Again, we add a {\it zero} auxiliary function $\phi(x,t)$ and 
consider the following $2\times 2$ system
 \begin{subequations}
 \label{aux-adv2}
\begin{align}
 u_t +f(u)_x+\epsilon  u_{xxx} & =0, &&\hspace{-2.8cm}  (x,t)\in I\times (0,T] ,\\
 \phi_t -\epsilon  \phi_{xxx} & =0, &&\hspace{-2.8cm}  (x,t)\in I\times (0,T].
\end{align}
\end{subequations}
To properly treat the nonlinear convection term, we shall first recall 
the {\it unique} square entropy-conserving flux, c.f. \cite{Tadmor03,ChenShu17},
for the scalar conservation law 
\[
 u_t + f(u)_x = 0,
\]
which reads:
\begin{align}
\label{eflux}
{f}_S(u_L, u_R) = \left\{
\begin{tabular}{l l}
 $\frac{\Psi(u_R)-\Psi(u_L)}{u_R-u_L}$& if $u_L\not=u_R$\\[1.5ex]
 $f(u_L)$& if $u_L=u_R$,
\end{tabular}
\right.
\end{align}
where $\Psi(u): = uf(u)-\int^u f(u)$ is the so-called {\it potential flux function} for the square entropy $U(u) = u^2/2$.
Note that the above choice of the nonlinear flux was used in \cite{Bona13}.

We denote the following operator
\begin{align}
\label{conv-linear}
 F_{u,j}((u_h, \phi_h), v_h) = &\;
 \intc{f(u_h)\,(v_h)_x} - \widehat{f}_{u,h} (v_h)^-|_{j+1/2}
+\widehat{f}_{u,h} (v_h)^+|_{j-1/2},
\end{align}
with the numerical flux  $\widehat{f}_{u,h} = f_S(u_h^-, u_h^+)$ being the square entropy conserving flux
given in \eqref{eflux}.
Then, the semi-discrete scheme for equations \eqref{aux-adv2} reads:
Find, for any time $t \in (0, T]$,
the unique function $(u_h, \phi_h) = (u_h(t), \phi_h(t))
\in V_h^k\times V_h^k$ such that
 \begin{subequations}
 \label{scheme:adv2}
\begin{align} 
\intc{(u_h)_tv_h}+F_{u,j}((u_h, \phi_h), v_h)  +D_{u,j}((u_h, \phi_h), v_h) =&\; 0,\\
\intc{(\phi_h)_t\psi_h} 
-D_{\phi,j}((u_h, \phi_h), \psi_h)=&\;0,
\end{align}
holds for all $(v_h, \psi_h)\in V_h^k\times V_h^k$ and all $\jj$.
\end{subequations}

It is easy to show that this method is energy-conserving in the sense of Theorem \ref{thm:adv1d}. 
An error analysis using the same energy argument as in Theorem \ref{thm:adv1d:err} leads to a suboptimal convergence rate of order $h^k$.
However, numerical results presented in the next section suggest that the method is optimally convergent.

\begin{remark}[Convective flux]
When $f(u) = u$, it is clear that the flux \eqref{eflux} is nothing but the central flux 
${f}_S(u_L, u_R) = \frac12(u_L+u_R)$. The DG method for linear scalar conservation law using a central flux 
is suboptimal in general, and only optimal for even polynomial degrees on uniform meshes, see \cite{FuShu18}.
Numerically results for the DG method using the flux \eqref{eflux}, not shown in this paper, 
for the Burgers equation ($u_t + uu_x =0$) also
indicates a similar convergence behavior, namely, suboptimal order $k$ for 
odd polynomial degrees, and optimal order $k+1$ for even polynomial degrees on uniform meshes.

However, the method \eqref{scheme:adv2} is numerically shown to be optimally convergent on general nonuniform meshes,
probably due to the interplay between the convection term and dispersion term. 
\end{remark}

\begin{remark}[Alternative auxiliary function]
We also tried to add the nonlinear convection term into the auxiliary function 
\[
 \phi_t-f(\phi)_x - \epsilon \phi_{xxx}=0,
\]
and discretize the resulting convective terms with the following operators
\begin{align*}
 {F}_{u,j}((u_h, \phi_h), v_h) = &\;
 \intc{f(u_h)\,(v_h)_x} - \tilde{f}_{u,h} (v_h)^-|_{j+1/2}
+\tilde{f}_{u,h} (v_h)^+|_{j-1/2},
\\
{F}_{\phi,j}((u_h, \phi_h), \psi_h) = &\;
 \intc{f(\phi_h)\,(\psi_h)_x} - \tilde{f}_{\phi,h} (\psi_h)^-|_{j+1/2}
+\tilde{f}_{\phi,h} (\psi_h)^+|_{j-1/2},
\end{align*}
with the numerical fluxes given by 
\begin{align*}
\tilde{f}_{u,h}|_{j-1/2} =&\; 
f_S(u_h^-, u_h^+) + \alpha_{j-1/2}\jump{\phi_h},
\\
\tilde{f}_{\phi,h}|_{j-1/2} =&\; 
f_S(\phi_h^-, \phi_h^+)+\alpha_{j-1/2}\jump{u_h},
\end{align*}
where 
\[
 \alpha_{j-1/2} = \max_{\min\{u_h^\pm, \phi_h^\pm\}\le v \le 
 \max\{u_h^\pm, \phi_h^\pm\}}\mathrm{sign}(\mean{u_h})|f'(v)|.
\]
This choice of fluxes mimic the linear scalar case in \cite{FuShu18}, and is energy-conserving.
While the resulting method is also numerically shown to be optimally convergent, 
we do observe a larger phase error of the method for long time simulations comparing with the method 
\eqref{scheme:adv2}. For this reason, we will stick with 
the formulation in \eqref{scheme:adv2}.
\end{remark}

\subsection{Time discretization}
In this paper, we will simply use the SSP-RK3 \cite{ShuOsher88} time discretization.
For the method of lines ODE 
\[
 (u_h)_t = L(u_h),
\]
the SSP-RK3 method is given by 
\begin{subequations}
\label{rk3}
\begin{align}
 u_h^{(1)} =& u_h^n + \Delta t L(u_h^n),\\
 u_h^{(2)} =& \frac34u_h^n +\frac14u_h^{(1)}+\frac14 \Delta t L(u_h^{(1)}),\\ 
 u_h^{n+1} =& \frac13u_h^n +\frac23 u_h^{(2)}+\frac23\Delta t L(u_h^{(2)}). 
 \end{align}
 \end{subequations}

This time discretization is explicit and dissipative.
For PDEs containing high order spatial derivatives as \eqref{eq:kdv}, 
explicit time discretization is subject to severe time step restriction for stability. 
We may also use the conservative, implicit Runge-Kutta collocation type methods 
discussed in \cite[Section 4]{Bona13}.
We do not pursue this issue since the focus of this paper is on the DG spatial discretization.

\section{Numerical results}
\label{sec:num}
We present numerical results to 
assess the performance of the proposed energy-conserving DG method.
We also compare results with the (dissipative) DG method of Cheng and Shu \cite{ChengShu08} denoted as (U),
and the (conservative) DG method of Bona et al. \cite{Bona13}, denoted as (C).
We denote the new method as (A). The RK3 time discretization \eqref{rk3} is used, with the time step size 
$\Delta t = CFL\, h^3$, for a sufficiently small CFL number to ensure stability.
All numerical simulation are performed using MATLAB.

\subsection*{Example 4.1: with no convection}
We consider the following equation
\begin{align}
\label{eq1}
 u_t + \frac{1}{4\pi^2}u_{xxx} = 0
\end{align}
on a unit interval $I=[0,1]$ with 
initial condition $u(x,0) = \sin(2\pi x)$, and a periodic boundary condition.
The exact solution is 
\begin{align*}
u(x, t) = \sin(2\pi (x+t)).
\end{align*}

Table \ref{table:adv1d} lists the numerical errors and their orders for 
the three DG methods at $T=10 \Delta t_0$, where $\Delta t_0$ is the time step size on the coarsest mesh.
We use $P^k$ polynomials with $2\le k\le 4$ on a nonuniform mesh which is a $10\%$ random perturbation of the 
uniform mesh.

From the table we conclude that, 
one can always observe optimal $(k+1)$th order of accuracy for both the variable $u_h$ (which approximate the solution $u$)
and $\phi_h$ (which approximate the {\it zero} function)
for the new energy-conserving DG method \eqref{scheme:adv1d}. This validates 
our convergence result in Theorem \ref{thm:adv1d:err}.
Moreover, the absolute value of the error is slightly smaller than
the optimal-convergent DG method (U) for all polynomial degrees.
We also observe suboptimal convergence for the method (C) for all 
polynomial degrees.
We specifically point out that while 
optimal convergence for the  method (C)
can be shown for even polynomial degrees on uniform meshes \cite{Bona13},
Table \ref{table:adv1d}  shows that such optimality no longer holds on nonuniform meshes, regardless of the polynomial degree.

\begin{table}[htbp]
\caption{\label{table:adv1d} 
The $L^2$-errors
and orders for Example 4.1 for the three DG methods
on a random mesh of $N$ cells. 
} \centering
\bigskip
\begin{tabular}{|c|c|cc|cc|cc|cc|}
\hline
 & &\multicolumn{2}{|c|}{(U)}&\multicolumn{2}{c|}{(C)}&\multicolumn{4}{c|}{(A)}\\
 \hline
 &  {$N$}      & $\|u-u_h\|$ & Order & $\|u-u_h\|$ & Order &  $\|u-u_h\|$ & Order 
 &  $\|\phi_h\|$  & Order\\
\hline
\multirow{4}{*}{$P^2$}
 & 10 & 6.57e-03 & 0.00 & 6.67e-03 & 0.00 & 8.73e-04 & 0.00 & 3.79e-03 & 0.00  \\
 & 20 & 2.21e-03 & 1.57 & 1.52e-03 & 2.13 & 1.77e-03 & -1.02 & 8.70e-04 & 2.12 \\ 
 & 40 & 2.90e-04 & 2.93 & 2.00e-04 & 2.93 & 1.63e-04 & 3.44 & 1.10e-04 & 2.99 \\ 
 & 80 & 3.61e-05 & 3.01 & 3.23e-05 & 2.63 & 1.93e-05 & 3.08 & 1.28e-05 & 3.10 \\ 
\hline
\multirow{4}{*}{$P^3$}
 & 10 & 1.94e-04 & 0.00 & 2.22e-04 & 0.00 & 7.99e-05 & 0.00 & 6.31e-05 & 0.00  \\
 & 20 & 2.04e-05 & 3.25 & 6.56e-05 & 1.76 & 2.19e-05 & 1.87 & 1.72e-05 & 1.88 \\ 
 & 40 & 1.29e-06 & 3.99 & 2.45e-05 & 1.42 & 1.06e-06 & 4.37 & 9.63e-07 & 4.16 \\ 
 & 80 & 7.99e-08 & 4.01 & 2.72e-06 & 3.17 & 5.93e-08 & 4.16 & 5.47e-08 & 4.14 \\ 
\hline
\multirow{4}{*}{$P^4$}
 & 10 & 4.65e-06 & 0.00 & 4.71e-06 & 0.00 & 1.42e-06 & 0.00 & 2.25e-06 & 0.00  \\
 & 20 & 2.13e-07 & 4.45 & 1.44e-07 & 5.03 & 1.67e-07 & 3.08 & 1.66e-07 & 3.76 \\ 
 & 40 & 5.88e-09 & 5.18 & 5.78e-09 & 4.64 & 4.17e-09 & 5.33 & 4.76e-09 & 5.12 \\ 
 & 80 & 1.82e-10 & 5.01 & 3.47e-10 & 4.06 & 1.46e-10 & 4.84 & 1.56e-10 & 4.93 \\ 
\hline
\end{tabular}
\end{table}

\subsection*{Example 4.2: with linear convection}
We consider the following equation
\begin{align}
\label{eq2}
 u_t - u_x + \frac{1}{4\pi^2}u_{xxx} = 0
\end{align}
on a unit interval $I=[0,1]$ with 
initial condition $u(x,0) = \sin(2\pi x)$, and a periodic boundary condition.
The exact solution is 
\begin{align*}
u(x, t) = \sin(2\pi (x+2t)).
\end{align*}

Table \ref{table:adv1d-D} lists the numerical errors and their orders for 
the three DG methods at $T=10 \Delta t_0$, where $\Delta t_0$ is the time step size on the coarsest mesh.
We use $P^k$ polynomials with $2\le k\le 4$ on a nonuniform mesh which is a $10\%$ random perturbation of the 
uniform mesh.

%
The convergence behavior of all three methods are similar to Example 4.1.
However, we do not have a proof of optimality of the methods (U) and (A).

\begin{table}[htbp]
\caption{\label{table:adv1d-D} 
The $L^2$-errors
and orders for Example 4.2 for the three DG methods on a random mesh of $N$ cells. 
} \centering
\bigskip
\begin{tabular}{|c|c|cc|cc|cc|cc|}
\hline
 & &\multicolumn{2}{|c|}{(U)}&\multicolumn{2}{c|}{(C)}&\multicolumn{4}{c|}{(A)}\\
 \hline
 &  {$N$}      & $\|u-u_h\|$ & Order & $\|u-u_h\|$ & Order &  $\|u-u_h\|$ & Order 
 &  $\|\phi_h\|$  & Order\\
\hline
\multirow{4}{*}{$P^2$}
 & 10 & 3.57e-03 & 0.00 & 4.59e-03 & 0.00 & 1.27e-03 & 0.00 & 3.10e-03 & 0.00  \\
 & 20 & 1.93e-03 & 0.88 & 1.13e-03 & 2.03 & 1.21e-03 & 0.07 & 1.08e-03 & 1.52 \\ 
 & 40 & 2.89e-04 & 2.74 & 1.82e-04 & 2.63 & 1.63e-04 & 2.90 & 1.15e-04 & 3.22 \\ 
 & 80 & 3.61e-05 & 3.00 & 2.67e-05 & 2.77 & 1.94e-05 & 3.07 & 1.30e-05 & 3.16 \\ 
\hline
\multirow{4}{*}{$P^3$}
 & 10 & 1.93e-04 & 0.00 & 2.22e-04 & 0.00 & 7.97e-05 & 0.00 & 6.30e-05 & 0.00  \\
 & 20 & 2.05e-05 & 3.24 & 6.66e-05 & 1.74 & 2.18e-05 & 1.87 & 1.72e-05 & 1.87 \\ 
 & 40 & 1.29e-06 & 3.99 & 2.45e-05 & 1.44 & 1.07e-06 & 4.35 & 9.65e-07 & 4.16 \\ 
 & 80 & 7.99e-08 & 4.01 & 2.71e-06 & 3.18 & 5.99e-08 & 4.16 & 5.54e-08 & 4.12 \\ 
\hline
\multirow{4}{*}{$P^4$}
 & 10 & 4.65e-06 & 0.00 & 4.72e-06 & 0.00 & 1.42e-06 & 0.00 & 2.25e-06 & 0.00  \\
 & 20 & 2.13e-07 & 4.45 & 1.37e-07 & 5.11 & 1.70e-07 & 3.06 & 1.69e-07 & 3.73 \\ 
 & 40 & 5.88e-09 & 5.18 & 5.77e-09 & 4.57 & 4.17e-09 & 5.35 & 4.76e-09 & 5.15 \\ 
 & 80 & 1.82e-10 & 5.01 & 3.47e-10 & 4.05 & 1.44e-10 & 4.85 & 1.55e-10 & 4.94 \\ 
\hline
\end{tabular}
\end{table}

\subsection*{Example 4.3: nonlinear convection: cnoidal wave}
We consider the following equation
\begin{align}
\label{eq3}
 u_t + \frac12(u^2)_x + \epsilon u_{xxx} = 0
\end{align}
on a unit interval $I=[0,1]$ with periodic boundary conditions.
Following \cite{Bona13},
we take $\epsilon = 1/24^2$ and choose the exact solution to be the following {\it cnoidal-wave} solution:
\[
 u(x,t) = a\,cn^2(4K(x-vt-x_0)),
\]
where $cn(z) = cn(z:m)$ is the Jacobi elliptic function with modulus $m\in (0,1)$ and the parameters
have the values $a=192m\epsilon K(m)^2$ and $v=64\epsilon (2m-1)K(m)^2$, $x_0=.5$.
The function $K=K(m)$ is the complete elliptic integral of the first kind. We use the value $m=0.9$. 

Table \ref{table:ac1d} lists the numerical errors and their orders for 
the three DG methods at $T=10 \Delta t_0$, where $\Delta t_0$ is the time step size on the coarsest mesh.
We use $P^k$ polynomials with $2\le k\le 4$ on a nonuniform mesh which is a $10\%$ random perturbation of the 
uniform mesh.

For Table \ref{table:ac1d}, similar conclusion as in Example 4.2  can be drawn.

\begin{table}[htbp]
\caption{\label{table:ac1d} 
The $L^2$-errors
and orders for Example 4.3 for the three DG methods
on a random mesh of $N$ cells. 
} \centering
\bigskip
\begin{tabular}{|c|c|cc|cc|cc|cc|}
\hline
 & &\multicolumn{2}{|c|}{(U)}&\multicolumn{2}{c|}{(C)}&\multicolumn{4}{c|}{(A)}\\
 \hline
 &  {$N$}      & $\|u-u_h\|$ & Order & $\|u-u_h\|$ & Order &  $\|u-u_h\|$ & Order 
 &  $\|\phi_h\|$  & Order\\
\multirow{4}{*}{$P^2$}
 & 10 & 6.57e-03 & 0.00 & 6.67e-03 & 0.00 & 8.73e-04 & 0.00 & 3.79e-03 & 0.00  \\
 & 20 & 2.21e-03 & 1.57 & 1.52e-03 & 2.13 & 1.77e-03 & -1.02 & 8.70e-04 & 2.12 \\ 
 & 40 & 2.90e-04 & 2.93 & 2.00e-04 & 2.93 & 1.63e-04 & 3.44 & 1.10e-04 & 2.99 \\ 
 & 80 & 3.61e-05 & 3.01 & 3.23e-05 & 2.63 & 1.93e-05 & 3.08 & 1.28e-05 & 3.10 \\ 
\hline
\multirow{4}{*}{$P^3$}
 & 10 & 5.51e-03 & 0.00 & 7.69e-03 & 0.00 & 6.30e-03 & 0.00 & 4.10e-03 & 0.00  \\
 & 20 & 1.44e-03 & 1.93 & 3.30e-03 & 1.22 & 1.67e-03 & 1.92 & 1.65e-03 & 1.32 \\ 
 & 40 & 9.99e-05 & 3.85 & 4.59e-04 & 2.84 & 9.48e-05 & 4.14 & 9.39e-05 & 4.13 \\ 
 & 80 & 6.17e-06 & 4.02 & 5.04e-05 & 3.19 & 5.37e-06 & 4.14 & 5.28e-06 & 4.15 \\ 
\hline
\multirow{4}{*}{$P^4$}
 & 10 & 2.20e-03 & 0.00 & 3.08e-03 & 0.00 & 1.33e-03 & 0.00 & 9.42e-04 & 0.00  \\
 & 20 & 7.28e-05 & 4.92 & 2.72e-04 & 3.50 & 1.40e-04 & 3.25 & 1.16e-04 & 3.02 \\ 
 & 40 & 2.71e-06 & 4.75 & 3.65e-06 & 6.22 & 3.68e-06 & 5.25 & 3.91e-06 & 4.90 \\ 
 & 80 & 8.47e-08 & 5.00 & 5.55e-07 & 2.72 & 1.28e-07 & 4.85 & 1.11e-07 & 5.14 \\ 
\hline
\end{tabular}
\end{table}

\subsection*{Example 4.4: long time simulation: cnoidal wave}
We consider the cnoidal wave problem in Example 4.3, and run the simulation with a much larger final  time.

We use both DG-$P^2$ space with 20 uniform cells, and DG-$P^3$ space with 10 uniform cells for the simulation.
The numerical results at time $T=5$ of the three DG methods are shown in  Figure \ref{fig:pw1d1}.
From this figure, 
we observe large dissipation and dispersion error for the method (U). We also observe 
poor resolution for the method (C), indicating that 
the mesh might be too coarse. Finally, we observe relatively the best result for the method (A) for both polynomial degrees.

\begin{figure}[ht!]
 \caption{Numerical solution at $T=5$ for Example 4.4.
 Top: method (U). Middle: method (C). Bottom: method (A).
 Left: DG-$P^2$ space, $20$ cells. Right: DG-$P^3$ space, $10$ cells.
 Solid line: numerical solution. Dashed line: exact solution.}
 \label{fig:pw1d1}
 \includegraphics[width=.45\textwidth]{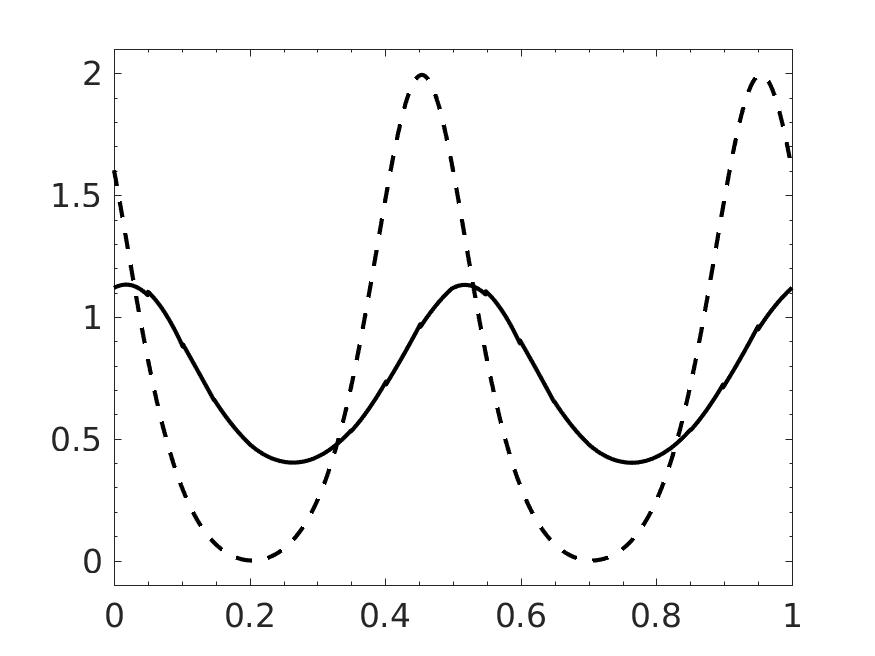}
 \includegraphics[width=.45\textwidth]{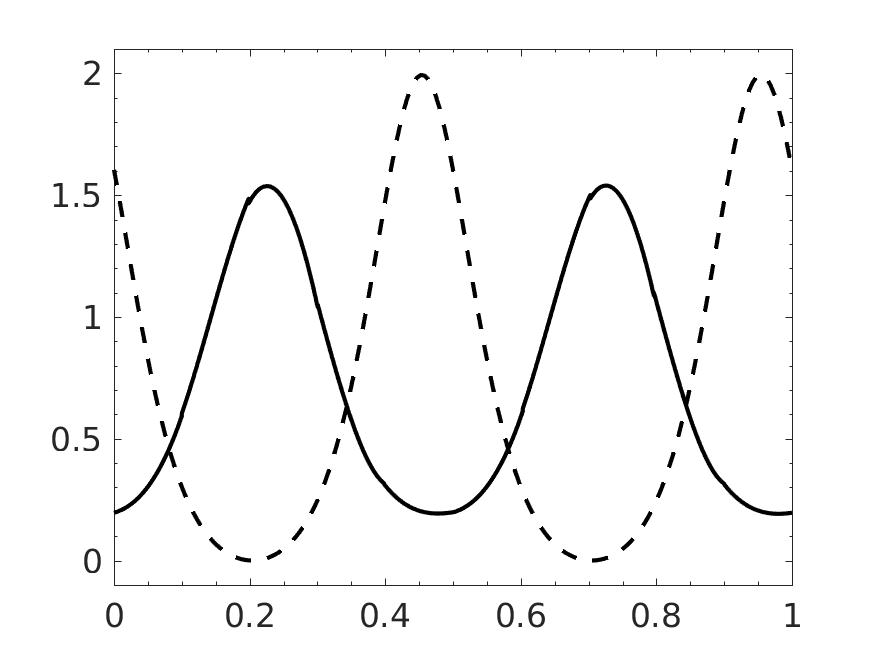}
 \includegraphics[width=.45\textwidth]{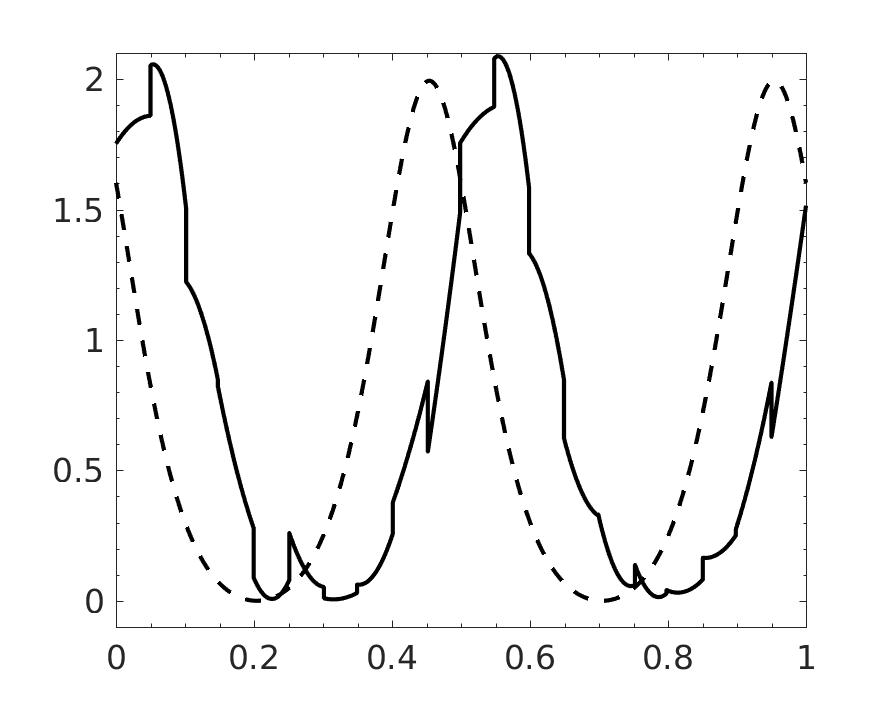}
 \includegraphics[width=.45\textwidth]{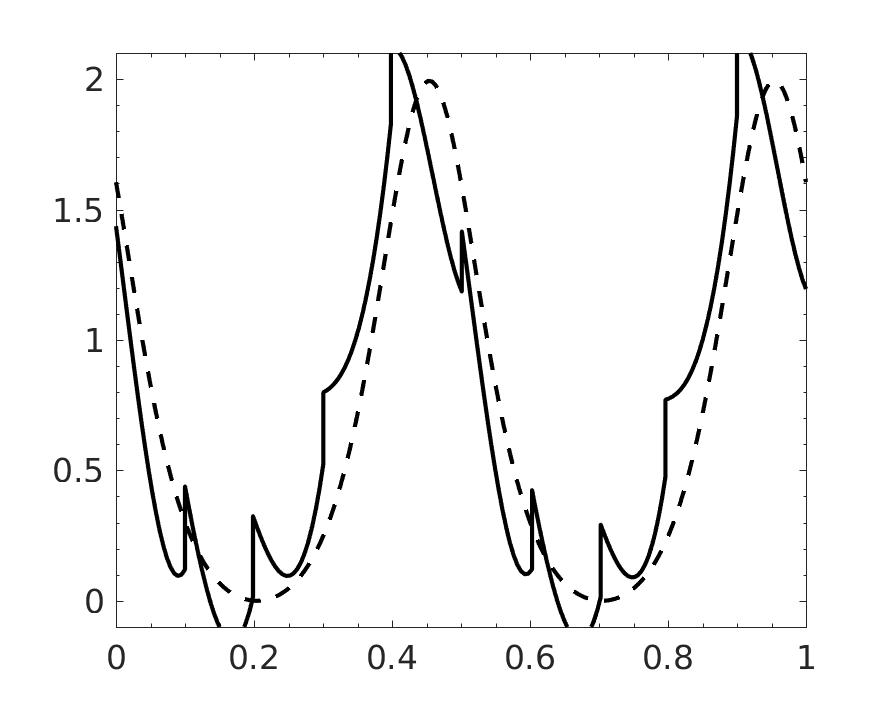}
 \includegraphics[width=.45\textwidth]{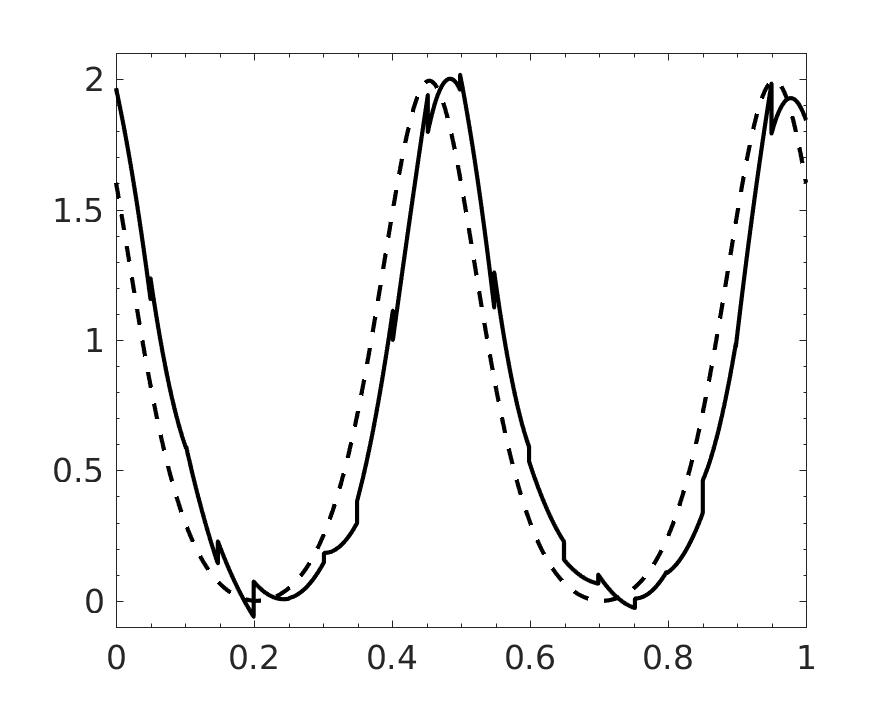}
 \includegraphics[width=.45\textwidth]{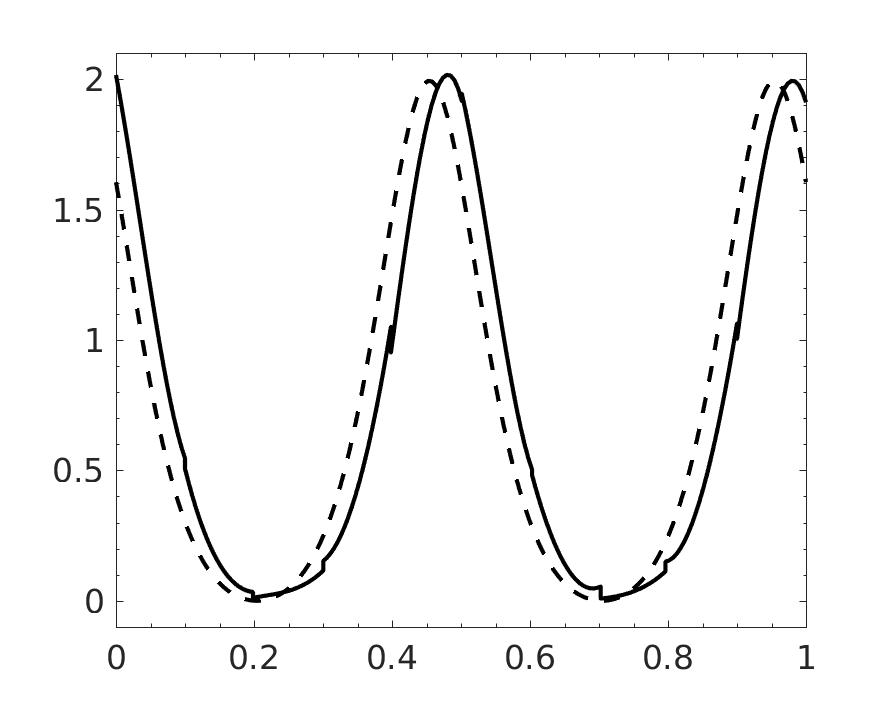}
\end{figure}

\section{Concluding remarks}
\label{sec:conclude}
In this paper, we have proposed an energy conserving DG method for the generalized KdV equation in one dimension.
The method is proven to be optimal convergent in the absence of the convection term, and 
is numerically shown to be optimally convergent on general non-uniform meshes.

Numerical comparison of the new method with the DG methods \cite{ChengShu08,Bona13} for long time simulations is also presented. 
The new method is found to be better than both methods \cite{ChengShu08, Bona13} 
in terms of the dissipation and dispersion errors.

Extension of the method to higher order odd-degree PDEs as those  considered 
in \cite{ChengShu08, XuShu10} poses no particular difficulty.
Extensions to multi-dimensions are also straight-forward.  Both will
be pursued in our future work.

\bibliographystyle{siam}

\begin{thebibliography}{10}

\bibitem{Bona13}
{\sc J.~L. Bona, H.~Chen, O.~Karakashian, and Y.~Xing}, {\em Conservative,
  discontinuous {G}alerkin-methods for the generalized {K}orteweg-de {V}ries
  equation}, Math. Comp., 82 (2013), pp.~1401--1432.

\bibitem{ChenShu17}
{\sc T.~Chen and C.-W. Shu}, {\em Entropy stable high order discontinuous
  {G}alerkin methods with suitable quadrature rules for hyperbolic conservation
  laws}, J. Comput. Phys., 345 (2017), pp.~427--461.

\bibitem{Chen15}
{\sc Y.~Chen, B.~Cockburn, and B.~Dong}, {\em A new discontinuous {G}alerkin
  method, conserving the discrete {$H^2$}-norm, for third-order linear
  equations in one space dimension}, IMA J. Numer. Anal., 36 (2016),
  pp.~1570--1598.

\bibitem{ChengShu08}
{\sc Y.~Cheng and C.-W. Shu}, {\em A discontinuous {G}alerkin finite element
  method for time dependent partial differential equations with higher order
  derivatives}, Math. Comp., 77 (2008), pp.~699--730.

\bibitem{Cockburn00}
{\sc B.~Cockburn, G.~E. Karniadakis, and C.-W. Shu}, {\em The development of
  discontinuous {G}alerkin methods}, in Discontinuous {G}alerkin methods
  ({N}ewport, {RI}, 1999), vol.~11 of Lect. Notes Comput. Sci. Eng., Springer,
  Berlin, 2000, pp.~3--50.

\bibitem{FuShu18}
{\sc G.~Fu and C.-W. Shu}, {\em {O}ptimal energy-conserving discontinuous
  {G}alerkin methods for linear symmetric hyperbolic systems}, arXiv:1804.10307
  [math.NA].
\newblock submitted on Apr. 2018.

\bibitem{ShuOsher88}
{\sc C.-W. Shu and S.~Osher}, {\em Efficient implementation of essentially
  nonoscillatory shock-capturing schemes}, J. Comput. Phys., 77 (1988),
  pp.~439--471.

\bibitem{Tadmor87}
{\sc E.~Tadmor}, {\em The numerical viscosity of entropy stable schemes for
  systems of conservation laws. {I}}, Math. Comp., 49 (1987), pp.~91--103.

\bibitem{Tadmor03}
\leavevmode\vrule height 2pt depth -1.6pt width 23pt, {\em Entropy stability
  theory for difference approximations of nonlinear conservation laws and
  related time-dependent problems}, Acta Numer., 12 (2003), pp.~451--512.

\bibitem{XuShu10}
{\sc Y.~Xu and C.-W. Shu}, {\em Local discontinuous {G}alerkin methods for
  high-order time-dependent partial differential equations}, Commun. Comput.
  Phys., 7 (2010), pp.~1--46.

\bibitem{YanShu02}
{\sc J.~Yan and C.-W. Shu}, {\em Local discontinuous {G}alerkin methods for
  partial differential equations with higher order derivatives}, in Proceedings
  of the {F}ifth {I}nternational {C}onference on {S}pectral and {H}igh {O}rder
  {M}ethods ({ICOSAHOM}-01) ({U}ppsala), vol.~17, 2002, pp.~27--47.

\bibitem{Yi13}
{\sc N.~Yi, Y.~Huang, and H.~Liu}, {\em A direct discontinuous {G}alerkin
  method for the generalized {K}orteweg-de {V}ries equation: energy
  conservation and boundary effect}, J. Comput. Phys., 242 (2013),
  pp.~351--366.

\end{thebibliography}

\end{document}